\documentclass[a4paper,10pt]{amsart}
\usepackage{amssymb,amsthm,amsmath}
\usepackage{ifthen}
\usepackage{graphicx}
\nonstopmode
 \numberwithin{equation}{section}
\newtheorem{thm}{Theorem}[section]
\newtheorem{cor}[thm]{Corollary}
\newtheorem{lem}[thm]{Lemma}
\newtheorem{prop}[thm]{Proposition}
\newtheorem{defn}[thm]{Definition}

\theoremstyle{definition}
\newtheorem{rmk}[thm]{Remark}

{\qed\bigskip}

\newcounter{alphabet}
\newcounter{tmp}

\ifx\undefined\bysame
\newcommand{\bysame}{\leavevmode\hbox to3em{\hrulefill}\,}
\fi

\begin{document}
\baselineskip=21pt
\markboth{} {}

\title[Approximation of the inverse frame operator and stability]
{Approximation of the inverse frame operator and stability of Hilbert$-$Schmidt frames}

\author{Anirudha Poria}

\address{Department of Mathematics, Indian Institute of Technology Guwahati, Assam 781039, India}
\email{a.poria@iitg.ernet.in}

\begin{abstract} 
In this paper we study the Hilbert$-$Schmidt frame (HS-frame) theory for separable Hilbert spaces. We first present some characterizations of HS-frames and prove that HS-frames share many important properties with frames. Then we show how the inverse of the HS-frame operator can be approximated using finite-dimensional methods. Finally we present a classical perturbation result and prove that HS-frames are stable under small perturbations.  
\end{abstract}
\date{\today}
\maketitle
\def\BC{{\mathbb C}} \def\BQ{{\mathbb Q}}
\def\BR{{\mathbb R}} \def\BI{{\mathbb I}}
\def\BZ{{\mathbb Z}} \def\BD{{\mathbb D}}
\def\BP{{\mathbb P}} \def\BB{{\mathbb B}}
\def\BS{{\mathbb S}} \def\BH{{\mathbb H}}
\def\BE{{\mathbb E}}  \def\BK{{\mathbb K}}
\def\BN{{\mathbb N}}
\def\g{{\mathcal{G}_j}}
\def\v{{\mathcal{V}_j}}
\def\a{{\mathcal{A}_j}}
\def\gtr{{\mathcal{G}_j^*}}
\def\gt{{\mathcal{\tilde{G}}_j}}
\def\ga{{\varGamma_j}}
\vspace{-.5cm}

\section{Introduction}
The concept of a frame in Hilbert spaces has been introduced in 1952 by Duffin and Schaeffer \cite{duf52}, in the context of nonharmonic Fourier series (see \cite{you01}). After the work of Daubechies et al. \cite{dau86} frame theory got considerable attention outside signal processing and began to be more broadly studied (see \cite{chr13, gro01}). A frame for a Hilbert space is a redundant set of vectors in Hilbert space which provides non-unique representations of vectors in terms of frame elements. The redundancy and flexibility offered by frames has spurred their application in several areas of mathematics, physics, and engineering such as wavelet theory, sampling theory, signal processing and many other well known fields.

Throughout this paper, $\BH$ and $\BK$ are separable Hilbert spaces, $\mathcal{L}(\BH)$ the algebra of all bounded linear operators on $\BH$, $I$ the identity operator on $\BH$, and $J$ is a countable index set. Recall that a family $\{f_j: j \in J \}$ in $\BH$ is called a {\it frame} for $\BH$, if there exist constants $0 < A \leq B < \infty $ such that for all $f \in \BH$ 
\begin{equation}
A \Vert f \Vert^2 \leq \sum_{j \in J } \vert \langle f,f_j \rangle \vert^2 \leq  B \Vert f \Vert^2.
\end{equation}
The constants $A$ and $B$ are called {\it lower} and {\it upper frame bounds}. We refer to \cite{dau92, han00, por16} for basic results on frames and \cite{ask01, kaf09, por15, sun06} for generalizations of frames.

Applications of frames, especially in the last decade, motivated the researcher to find some generalization of frames. Hilbert$-$Schmidt frames, or simply HS-frames were introduced in \cite{sad12} as a class of von Neumann$-$Schatten $p$-frames, which generalized all the existing frames such as $g$-frames \cite{sun06}, bounded quasi-projectors \cite{for04}, frames of subspaces \cite{cas04}, pseudo-frames \cite{li04}, oblique frames \cite{chr04}, outer frames \cite{ald04}, and time-frequency localization operators \cite{dor06}. Recent applications of HS-frames (see \cite{por17}), inspired us to study HS-frames in Hilbert spaces. It is well known that $g$-frames and $g$-Riesz bases in Hilbert spaces have some properties similar to those of frames and Riesz bases, but not all the properties are similar, e.g., exact $g$-frames are not equivalent to $g$-Riesz bases (see \cite{sun06, sun07}). The natural question to ask is: which properties of the frame, or the $g$-frame may be extended to the HS-frame for a Hilbert space? In Section \ref{sec2}, we investigate this problem. We introduce the synthesis operator for the HS-frame and using the synthesis operator, we establish some necessary and sufficient conditions for a HS-Bessel sequence, a HS-frame, and a HS-Riesz basis in a Hilbert space. We also characterize HS-frames from the point of view of operator theory and discuss the relation between a HS-frame and a HS-Riesz basis.

The reconstruction formula for a frame allows every element in the Hilbert space to be written as a linear combination of the frame elements, with frame coefficients. Calculations of those coefficients require knowledge of the inverse frame operator. But in practice it is very difficult to invert the frame operator if the Hilbert space is infinite dimensional. Calculations of the inverse frame operator for HS-frames in infinite dimensional Hilbert space is also very difficult. Christensen introduced the projection method in \cite{chr93} and the strong projection method in \cite{chr96} to approximate the frame coefficients. Following Christensen in \cite{cass97, cas2000, chr00}, the authors proved that the inverse frame operator can be approximated arbitrarily closely using finite-dimensional linear algebra. Using similar methods, the authors of \cite{abd11} proved approximation results for inverse $g$-frame operators. In Section \ref{sec3}, we derive a method to approximate the inverse HS-frame operator in the strong operator topology, using finite subsets of the HS-frame.

Given a family $\{g_j: j \in J \} \subseteq \BH$ which is close to the frame or Riesz basis $\{f_j: j \in J \} \subseteq \BH$, finding conditions to ensure that $\{g_j: j \in J \}$ is also a frame or Riesz basis is called the stability problem. This problem is important in practice, so it has received much attentions and is therefore studied widely by many authors (see \cite{chr95, fav95, naj08, sun07}). Since frames can be characterized in terms of operators, many results on perturbations of frames can also be characterized from the operator point of view (see \cite{cas97, guo13}). In Section \ref{sec4}, we study the stability of HS-frames. We first present a classical perturbation result of HS-frames. Then we give other perturbations of HS-frames.
\section{Characterization of Hilbert$-$Schmidt frames}\label{sec2}
Let us denote $\{ \BK_j: j \in J \} \subset \BK$ as a sequence of Hilbert spaces and $\mathcal{L}(\BH,\BK_j)$ the collection of all bounded linear operators from $\BH$ to $\BK_j.$ Note that for any sequence $\{ \BK_j: j \in J \}$, we can always find a larger space $\BK$ containing all the Hilbert space $\BK_j$ by setting $\BK=\bigoplus_{j \in J} \BK_j$. The notion of a frame was extended to a $g$-frame by Sun \cite{sun06}. First we recall the definition of a $g$-frame. 
\begin{defn} \cite{sun06}
A family $\{ \Lambda_j \in \mathcal{L}(\BH,\BK_j):j \in J \}$ is called a generalized frame, or simply a $g$-frame, for $\BH$ with respect to $\{ \BK_j: j \in J \}$ if there are two constants $A, B>0$ such that for all $f \in \BH$
\begin{equation}
A \Vert f \Vert^2 \leq \sum_{j \in J }  \Vert \Lambda_j(f)  \Vert^2 \leq  B \Vert f \Vert^2.
\end{equation} 
\end{defn}
Let $\mathcal{L}(\BH)$ denote the $C^*$-algebra of all bounded linear operators on a complex separable Hilbert space $\BH$. For a compact operator $T \in \mathcal{L}(\BH)$, the eigenvalues of the positive operator $\vert T \vert = (T^*T)^{1/2}$ are called the singular values of $T$ and denoted by $s_j(T)$. We arrange the singular values $s_j(T)$ in a decreasing order and these are repeated according to multiplicity, that is, $s_1(T)\geq s_2(T) \geq ... \geq 0$. For $1 \leq p < \infty$, the von Neumann$-$Schatten p-class $C_p$ is defined to be the set of all compact operators $T$ for which 
\begin{equation}
\Vert T \Vert_p=(\tau \vert T \vert^p)^{\frac{1}{p}}=\bigg( \sum_{j=1}^{\infty} s_j^p(T) \bigg)^{\frac{1}{p}} < \infty,
\end{equation} 
where $\tau$ is the usual trace functional defined as $\tau(T)=\sum_{e \in E}\langle T(e),e \rangle$, and $E$ is any orthonormal basis of $\BH$. For $p=\infty$, let $C_{\infty}$ denote the class of all compact operators with $\Vert T \Vert_{\infty}=s_1(T)< \infty$. For more information about a von Neumann$-$Schatten ${p}$-class see \cite{rin71}. We recall that $C_2$ is a Banach space with respect to $\Vert . \Vert_2$, and also it is a Hilbert space with the inner product defined by $\big[ T,S \big]_{\tau}=\tau(S^*T)$. Also, $C_2$ is called the Hilbert$-$Schmidt class. An operator $T \in \mathcal{L}(\BH)$ belongs to the Hilbert$-$Schmidt class if and only if $\Vert T \Vert^2_{HS}:=\sum_{j \in J} \Vert Te_j \Vert^2 < \infty,$ where $\{ e_j \}_{j  \in J}$  is any orthonormal basis for $\BH$. Notice that $\Vert T \Vert_{HS}=\Vert T \Vert_2.$

\begin{defn} \cite{sad12}
A family $\{\g : j \in J\}$ of bounded linear operators from $\BH$ to $C_2 \subseteq \mathcal{L}(\BK)$ is said to be a Hilbert$-$Schmidt frame, or simply a HS-frame for $\BH$ with respect to $\BK$, if there exist constants $A, B >0$ such that for all $f \in \BH$
\begin{equation}\label{eq4}
A \Vert f \Vert^2 \leq \sum_{j \in J }  \Vert \g(f)  \Vert_2^2 \leq  B \Vert f \Vert^2.
\end{equation}
\end{defn}
If the right-hand side of $(\ref{eq4})$ holds, it is said to be a $HS$-$Bessel \; sequence$ with bound $B$. If $\{f \in \BH: \g(f)=0, \forall j \in J\}=\{0\}$, then $\{\g : j \in J\}$ is called $HS$-$complete$. If $\{\g : j \in J\}$ is HS-complete and there are positive constants $A$ and $B$ such that for any finite subset $J_1 \subset J$ and $\a \in C_2, j \in J_1,$
\begin{equation}\label{eq8}
A \sum_{j \in J_1 } \Vert \a \Vert^2 \leq \bigg\Vert \sum_{j \in J_1 } \g^* (\a)  \bigg\Vert^2 \leq  B \sum_{j \in J_1 } \Vert \a \Vert^2,
\end{equation}
then $\{\g : j \in J\}$ is called a $HS$-$Riesz \; basis$ for $\BH$ with respect to $\BK$.

For $x,y \in \BH$, we define the operator $x \otimes y : \BH \rightarrow \BH$ by
\[ (x \otimes y)(z)=\langle z,y \rangle x, \;\; z \in \BH.   \]
It is obvious that $\Vert x \otimes y \Vert= \Vert x \Vert \Vert y \Vert $, and if $x$ and $y$ are non-zero, then the rank of $x \otimes y$ is one. If $x,y,z,w \in \BH$, then the following equalities are easily verified:
\begin{eqnarray*}
(x \otimes y)(z \otimes w) & = & \langle z,y \rangle (x \otimes w) \\ (x \otimes y)^* &=& y \otimes x. 
\end{eqnarray*}
Let $y_0 \in \BK$ be an unit vector, the operator $\mathcal{W}: \BK \rightarrow C_2 \subseteq \mathcal{L}(\BK)$ defined by $\mathcal{W}x=x \otimes y_0$ is a linear isometry since $ \Vert \mathcal{W}x  \Vert_2 = \Vert x \otimes y_0 \Vert_2= \Vert x \Vert $. So we can consider $\BK$ as subspace of $C_2$, and hence it is a subspace of $\mathcal{L}(\BK)$.

\begin{lem}\label{lem00}
\cite{sad12} Let $\{\Lambda_j :j \in J\}$ be a $g$-frame for $\BH$ with respect to $\{ \BK_j: j \in J \}$. Then $\{\Lambda_j :j \in J\}$ is a HS-frame for $\BH$ with respect to $\BK=\bigoplus\limits_{j \in J}\BK_j.$
\end{lem}
In \cite{sun06}, Sun has shown that bounded quasi-projectors \cite{for04}, frames of subspaces \cite{cas04}, pseudo-frames \cite{li04}, oblique frames \cite{chr04}, outer frames \cite{ald04}, and time-frequency localization operators \cite{dor06} are special classes of $g$-frames. Hence, Lemma \ref{lem00} implies that each of these classes is also a class of HS-frames.
\begin{rmk}
Each $\g \in \mathcal{L}(\BH, C_2)$ is an operator-valued function. So HS-frames $\{ \g : j \in J \}$, are an operator-valued frame. In particular, if we consider $\BK_j \subseteq \BK \subseteq C_2 \subseteq \mathcal{L}(\BK)$, then $g$-frames for $\BH$ with respect to $\{\BK_j: {j \in J} \}$ can be considered as HS-frames for $\BH$ with respect to $\BK$. Thus HS-frames share many useful properties with $g$-frames.
\end{rmk}
Suppose $\{\mathcal{X}_j : j \in J\}$ is a collection of normed spaces. Then $\prod\{\mathcal{X}_j : j \in J\}$ is a vector space if the linear operations are defined coordinatewise. Define
\[\bigoplus \mathcal{X}_j \equiv \Big\{ x \in \prod_{j \in J} \mathcal{X}_j: \Vert x \Vert=( \sum_{j \in J} \Vert x_j \Vert^2 )^{1/2} < \infty \Big\}. \]
with the inner product given by $\langle x, y \rangle=\sum_{j \in J} \langle x_j,y_j\rangle$. It is known that $\bigoplus \mathcal{X}_j$ is a Hilbert space if and only if so is each $\mathcal{X}_j$. 

Now we define the synthesis operator for a HS-frame. For this purpose, we first show that the series appearing in the definition of a synthesis operator converges unconditionally. So we need the next lemma.
\begin{lem}\label{lem1}
Let $\{\g : j \in J\}$ be a HS-Bessel sequence for $\BH$ with bound $B$. Then for each sequence $\{\a\}_{j \in J} \in \bigoplus C_2,$ the series $\sum_{j \in J} \g^*(\a)$ converges unconditionally.
\end{lem}
\begin{proof}
Let $J_1 \subseteq J$ with $|J_1| < \infty,$ then
\begin{eqnarray*}
\bigg\Vert \sum_{j \in J_1} \g^*(\a) \bigg\Vert &=& \sup_{h \in \BH, \; \Vert h \Vert=1} \bigg| \bigg\langle \sum_{j \in J_1} \g^*(\a), h \bigg\rangle \bigg| \\
& \leq & \bigg( \sum_{j \in J_1} \Vert \a \Vert^2 \bigg)^{1/2} \sup_{h \in \BH, \; \Vert h \Vert=1} \bigg( \sum_{j \in J_1} \Vert \g(h) \Vert^2 \bigg)^{1/2} \\
& \leq & \sqrt{B} \bigg( \sum_{j \in J_1} \Vert \a \Vert^2 \bigg)^{1/2}.
\end{eqnarray*}
It follows that $\sum_{j \in J} \g^*(\a)$ is weakly unconditionally Cauchy and hence unconditionally convergent in $\BH$.
\end{proof}
\begin{defn}
Let $\{\g : j \in J\}$ be a HS-frame for $\BH$. Then the synthesis
operator for $\{\g : j \in J\}$ is the operator $ T:\bigoplus C_2 \rightarrow \BH$ defined by $T(\{ \a \}_{j \in J})=\sum_{j \in J} \g^*(\a).$
\end{defn}
The adjoint $T^*$ of the synthesis operator is called the $analysis \; operator$. The following lemma provides a formula for the analysis operator.
\begin{lem}
Let $\{\g : j \in J\}$ be a HS-frame for $\BH$. Then the analysis operator $T^*: \BH \rightarrow \bigoplus C_2$, given by $T^*(f)=\{ \g(f) \}_{j \in J}$ is well defined.
\end{lem}
\begin{proof}
Let $f \in \BH$ and $\{\a\}_{j \in J} \in \bigoplus C_2$. Then
\begin{eqnarray*}
\langle T^*(f), \{\a\}_{j \in J} \rangle &=& \langle f, T \{\a\}_{j \in J} \rangle = \bigg\langle f, \sum_{j \in J} \g^*(\a) \bigg\rangle \\  &=& \sum_{j \in J} \big[ \g (f), \a \big]_{\tau}= \big\langle \{ \g (f)\}_{j \in J}, \{\a\}_{j \in J} \big\rangle.
\end{eqnarray*} 
Hence $T^*(f)=\{ \g(f) \}_{j \in J}$ is well defined.
\end{proof}
In the following proposition, we characterize the HS-Bessel sequence in terms of the synthesis operator.
\begin{prop}\label{pro1}
A sequence $\{\g : j \in J\} \subseteq \mathcal{L}(\BH, C_2)$ is a HS-Bessel sequence for $\BH$ with bound $B$ if and only if the synthesis operator $T$ is a well defined bounded operator with $\Vert T \Vert \leq \sqrt{B}$.
\end{prop}
\begin{proof}
Let $\{\g : j \in J\}$ is a HS-Bessel sequence for $\BH$ with bound $B$. Then by Lemma \ref{lem1}, $T$ is a well defined bounded operator with $\Vert T \Vert \leq \sqrt{B}$.

Conversely, let $T$ be a well defined and $\Vert T \Vert \leq \sqrt{B}$. Let $J_1 \subseteq J$ with $|J_1| < \infty,$ then  
\[ \sum_{j \in J_1} \Vert \g(f) \Vert^2=\sum_{j \in J_1} \langle \gtr \g (f),f \rangle= \big\langle T(\{ \g (f) \}_{j \in J_1}),f \big\rangle \leq \Vert T \Vert \Vert \{\g(f)\}_{j \in J_1} \Vert \Vert f \Vert, \; \forall f \in \BH. \] Therefore 
\[ \sum_{j \in J_1} \Vert \g(f) \Vert^2 \leq \Vert T \Vert \bigg( \sum_{j \in J_1} \Vert \g(f) \Vert^2 \bigg)^{1/2}\Vert f \Vert \leq \Vert T \Vert^2 \Vert f \Vert^2 \leq B \Vert f \Vert^2.  \] It follows that $\{\g : j \in J\}$ is a HS-Bessel sequence for $\BH$ with bound $B$.
\end{proof}
\begin{defn}
Let $\{\g : j \in J\}$ be a HS-frame for $\BH$. Then the HS-frame operator for $\{\g : j \in J\}$ is the operator $S:\BH \rightarrow \BH$ defined by $Sf=TT^*f=\sum\limits_{j \in J} \gtr\g(f)$.
\end{defn}
If $\{\g : j \in J\}$ is a HS-frame with bounds $A$ and $B$, then for any $f \in \BH$ we have
\[ \langle Sf, f \rangle=\bigg\langle \sum\limits_{j \in J} \gtr\g(f),f \bigg\rangle=\sum\limits_{j \in J} \big[ \g(f),\g(f) \big]_{\tau}=\sum\limits_{j \in J} \Vert \g(f) \Vert^2. \]
Hence \[ A \langle f,f \rangle \leq \langle Sf,f\rangle \leq  B \langle f,f \rangle,\mathrm{\; i.e.,\;} AI \leq S \leq BI. \]
Therefore S is a bounded, invertible and positive self-adjoint operator. Also, the following reconstruction formula holds for all $f \in \BH$ 
\begin{equation}
f=SS^{-1}f=S^{-1}Sf=\sum_{j \in J} \gtr \g S^{-1}f=\sum\limits_{j \in J} S^{-1}\gtr \g f.
\end{equation}
Moreover, $\{\g S^{-1}: j \in J \}$ is a HS-frame with bounds $B^{-1}$ and $A^{-1}$. We call $\{\gt = \g S^{-1}: j \in J \}$ the $canonical \;dual \; HS$-$frame$ of $\{\g : j \in J\}$. A HS-frame $\{\v : j \in J\}$ is called an $alternate \; dual \; HS$-$frame$ of $\{\g : j \in J \}$ if for all $f \in \BH$ the following identity holds: 
\begin{equation}\label{eq01}
f=\sum\limits_{j \in J}\gtr \v f=\sum\limits_{j \in J} \v^{*}\g f .
\end{equation}
The following result provides a connection between a HS-frame and a HS operator.
\begin{prop}
Let $S \in \mathcal{L}(\BH)$ be a HS-frame operator. Then, $S$ is a Hilbert$-$Schmidt operator if and only if $\BH$ is finite-dimensional.
\end{prop}
\begin{proof}
Let $\{e_n\}_{n \in J}$ be an orthonormal basis for $\BH$. Using Lemma \ref{lem1}, we get
\[
\Vert S \Vert^2_{HS}=\sum_{n \in J} \Vert Se_n \Vert^2=\sum_{n \in J} \bigg\Vert \sum_{j \in J} \g^* \g (e_n) \bigg\Vert^2 \leq B \sum_{n \in J} \sum_{j \in J} \Vert \g (e_n) \Vert^2 \leq B \sum_{n \in J} B \Vert e_n \Vert^2.
\]
If dim $\BH=$ card $J < \infty$, we have $\Vert S \Vert^2_{HS} \leq B^2 \mathrm{\; card \;} J < \infty.$

Conversely, let $S$ be a Hilbert$-–$Schmidt operator. Since Hilbert$-$Schmidt operators are compact, $S$ is compact. Also, $S$ is invertible on $\BH$. Thus $SS^{-1}=I$ implies that the identity $I$ must be a compact operator. Hence dim $\BH< \infty.$
\end{proof}
\begin{rmk}
Since $\BH$ is an infinite-dimensional Hilbert space, the HS-frame operator $S$ cannot be a Hilbert-–Schmidt operator.
\end{rmk}
\begin{lem}\label{lem2}
\cite{chr13} Suppose that $U: \BK \rightarrow \BH$ is a bounded surjective operator.
Then there exists a bounded operator (called the pseudo-inverse of $U$) $U^{\dagger}: \BH \rightarrow \BK$ for which \[UU^{\dagger}f=f, \;\; \forall f \in \BH. \]

If $U$ is a bounded invertible operator, then $U^{\dagger}=U^{-1}$.
\end{lem}
In the following proposition we establish a relationship between a HS-frame and the associated synthesis operator.
\begin{prop}\label{pro2}
A sequence $\{\g : j \in J\} \subseteq \mathcal{L}(\BH, C_2)$ is a HS-frame for $\BH$ if and only if the synthesis operator $T$ is a well defined, bounded and surjective operator.
\end{prop}
\begin{proof}
If $\{\g : j \in J\}$ is a HS-frame for $\BH$, then $S=TT^*$ is invertible. So $T$ is surjective. Conversely, let $T$ be well defined, bounded and surjective operator. Then by Proposition \ref{pro1}, the sequence $\{\g : j \in J\}$ is a HS-Bessel sequence for $\BH$. Since $T$ is surjective, by Lemma \ref{lem2}, there exists an operator $T^{\dagger}:\BH \rightarrow \bigoplus C_2$ such that $TT^{\dagger}=I.$ Hence $(T^{\dagger})^*T^*=I$. Then for all $f \in \BH$,  
\begin{equation*}
\Vert f \Vert^2 \leq \Vert (T^{\dagger})^* \Vert^2 \Vert T^*f \Vert^2=\Vert T^{\dagger} \Vert^2 \Vert T^*f \Vert^2= \Vert T^{\dagger} \Vert^2 \sum_{j \in J} \Vert \g(f)\Vert^2. 
\end{equation*} 
It follows that $\{\g : j \in J\}$ is a HS-frame for $\BH$ with lower HS-frame bound
$\Vert T^{\dagger} \Vert^{-2}$ and upper HS-frame bound $\Vert T \Vert^2$.
\end{proof}
Now we establish the relation between a HS-frame and a HS-Riesz basis. We first establish the following lemma.
\begin{lem}\label{lem4}
A sequence $\{\g : j \in J\} \subseteq \mathcal{L}(\BH, C_2)$ is a HS-Riesz basis for $\BH$ with bounds $A$ and $B$ if and only if the synthesis operator $T$ is a linear homeomorphism such that
\begin{equation}\label{eq7}
 A \sum_{j \in J } \Vert \a \Vert^2 \leq \Vert T (\{ \a \}_{j \in J}) \Vert^2 \leq  B \sum_{j \in J } \Vert \a \Vert^2, \;\; \forall  \{ \a \}_{j \in J} \in \bigoplus C_2.
\end{equation}
\end{lem}
\begin{proof}
If $\{\g : j \in J\}$ is a HS-Riesz basis for $\BH$ with bounds $A$ and $B$, then from the definition of HS-Riesz bases, the synthesis operator $T$ is a bounded, injective operator with the closed range $T(\bigoplus C_2)$ and $\Vert T \Vert \leq \sqrt{B}$. So, from Proposition \ref{pro1}, the sequence $\{\g : j \in J\}$ is a HS-Bessel sequence for $\BH$. Let $f \in [T(\bigoplus C_2)]^{\bot}$, then $\{\g(f)\} _{j \in J} \in \bigoplus C_2$. Hence we get
\[ 0= \langle T(\{\g(f)\} _{j \in J}),f \rangle=\bigg\langle \sum_{j \in J} \gtr \g (f),f \bigg\rangle=\sum_{j \in J} [ \g(f), \g(f) ]_{\tau}= \sum_{j \in J} \Vert \g(f) \Vert^2. \]
It implies that $\g(f)=0,$ for all $j \in J$. Since $\{\g: j \in J \}$ is HS-complete, we obtain $f=0,$ which proves $T(\bigoplus C_2)=\BH$. Hence $T$ is a linear homeomorphism. Also, from Equation (\ref{eq8}), for every $\{ \a \}_{j \in J} \in \bigoplus C_2$ we obtain
\[  A \sum_{j \in J } \Vert \a \Vert^2 \leq \Vert T (\{ \a \}_{j \in J}) \Vert^2 \leq  B \sum_{j \in J } \Vert \a \Vert^2. \]

Conversely, If $T$ is a linear homeomorphism satisfying (\ref{eq7}), then by Proposition \ref{pro2}, we find that $\{\g : j \in J\}$ is a HS-frame for $\BH$ with bounds $\Vert T^{\dagger} \Vert^{-2}$ and $\Vert T \Vert^2$. If $\g(f)=0$ for $f \in \BH$ and all $j \in J$, then $\Vert f \Vert^2 \leq \Vert T^{\dagger} \Vert^2 \sum_{j \in J} \Vert \g(f)\Vert^2=0$ implies $f=0.$ Thus $\{\g : j \in J\}$ is a HS-complete. Now by the definition of HS-Riesz bases and the inequalities (\ref{eq7}), we conclude that $\{\g : j \in J\}$ is a HS-Riesz basis for $\BH$ with bounds $A$ and $B$. This completes the proof.
\end{proof} 
\begin{thm}\label{th3}
Let $\{\g : j \in J\} \subseteq \mathcal{L}(\BH, C_2)$. Then the following are equivalent:

$(1)$ The sequence $\{\g : j \in J\}$ is a HS-Riesz basis for $\BH$ with bounds $A$ and $B$.

$(2)$ The sequence $\{\g : j \in J\}$ is a HS-frame for $\BH$ with bounds $A$ and $B$, and $\{\g : j \in J\}$ is an $\bigoplus C_2$-linearly independent family, i.e., if $\sum_{j \in J} \gtr(\a)=0$ for $\{ \a \}_{j \in J} \in \bigoplus C_2$, then $\a=0$ for all $j \in J$.
\end{thm}
\begin{proof}
$(1)\Rightarrow (2)$ From Lemma \ref{lem4}, the operator $T$ is a linear homeomorphism with $\Vert T^{\dagger} \Vert^{2}=\Vert T^{-1} \Vert^2 \leq \frac{1}{A}$ and $\Vert T \Vert^2 \leq B$. Thus the operator $T$ is surjective with  $ \Vert T^{\dagger} \Vert^{-2} \geq A$ and 
\begin{equation}\label{eq9}
\mathrm{ker} \; T= \bigg\{ \{\a \}_{j \in J} \in \bigoplus C_2 : T(\{\a \}_{j \in J})=\sum_{j \in J} \gtr (\a)=0 \bigg\}=\{0\}.
\end{equation}
It follows that $\{\g : j \in J\}$ is an $\bigoplus C_2$-linearly independent family. Hence by Proposition \ref{pro2}, the statement $(1)$ implies $(2)$.

$(2)\Rightarrow (1)$ From Proposition \ref{pro2} and (\ref{eq9}), the operator $T$ is a linear homeomorphism with $\Vert T \Vert^2 \leq B$, so is the adjoint $T^*$. Since $\{\g : j \in J\}$ is a HS-frame for $\BH$ with bounds $A$ and $B$, $\Vert T^*(f) \Vert^2= \sum_{j \in J} \Vert \g(f) \Vert^2 \geq A \Vert f \Vert^2$. So, $\Vert T^{-1} \Vert^2=\Vert (T^*)^{-1} \Vert^2 \leq A^{-1}$. Hence for all $\{\a \}_{j \in J} \in \bigoplus C_2$, we have 
\begin{eqnarray*}
&& \Vert T(\{\a \}_{j \in J}) \Vert^2 \leq \Vert T \Vert^2 \Vert \{\a \}_{j \in J} \Vert^2 \leq B \sum_{j \in J} \Vert \a \Vert^2, \\
&& \Vert \{\a \}_{j \in J} \Vert^2 = \Vert T^{-1}T(\{\a \}_{j \in J}) \Vert^2 \leq \Vert T^{-1} \Vert^2 \Vert T(\{\a \}_{j \in J}) \Vert^2 \leq \frac{1}{A} \Vert T(\{\a \}_{j \in J}) \Vert^2.
\end{eqnarray*}
From Lemma \ref{lem4}, the statement $(2)$ implies $(1)$. This completes the proof.
\end{proof}
\section{Approximation of the inverse HS-frame operator}\label{sec3}
In this section, $\BH$ denotes a finite dimensional Hilbert space and let $\{J_n\}_{n=1}^{\infty}$ be a family of finite subsets of $J$ such that $J_1 \subseteq J_2 \subseteq ... \subseteq J_n \nearrow J.$ Given a family $\{\g: j \in J \} \subseteq \mathcal{L}(\BH, C_2)$, we define the space $\BH_n=\mathrm{span} \{ \g^*(C_2): j \in J_n \} $. Then it is easy to see that $\{\g: j \in J_n \}$ is a HS-frame for $\BH_n$. The HS-frame operator for $\{\g: j \in J_n \}$ is 
\[S_n: \BH_n \to \BH_n, \quad S_nf=\sum_{j \in J_n} \g^* \g f.\]
We show that the inverse HS-frame operator $S^{-1}$ can be approximated by operators $S_n^{-1}$ using finite dimensional methods. Here $S_n$ is an operator on the finite dimensional space $\BH_n$. In the following theorem, we generalize Theorem 3.1 in \cite{chr93} from the setting of Hilbert space frames to HS-frames.

\begin{thm}
Let $\{\g: j \in J \}$ be a  HS-frame for $\BH$ with bounds $A$ and $B$. Then for every $f,g \in \BH$ 
\begin{equation}\label{eq11}
\langle g, S_n^{-1} \g^* \g f \rangle \to \langle g, S^{-1} \g^* \g f \rangle \quad \mathrm{as} \quad n \to \infty,
\end{equation}
if and only if for every $j \in J$ and every $f \in \BH$ there exists a constant $c_j$ such that
\begin{equation}\label{eq12}
\| S_n^{-1} \g^* \g f \| \leq c_j, \quad \forall n \mathrm{\; such \; that \;} j \in J_n.
\end{equation}
\end{thm}
\begin{proof}
Assume that (\ref{eq11}) is satisfied. Fix $f \in \BH$ and $j \in J$. For every $n$ with $ j \in J_n$, define 
\[ F_n : \BH \to \BC, \quad F_n(g)=\langle g, S_n^{-1} \g^* \g f \rangle.\]
Then each $F_n$ is continuous, and by (\ref{eq11}) the family $\{F_n\}$ converges pointwise. By Banach Steinhaus theorem there is a constant $c_j$ such that $\|F_n\|=\| S_n^{-1} \g^* \g f \| \leq c_j$ for all $n$.

Conversely, suppose (\ref{eq12}) is satisfied. Let $f \in \BH$. Fix a $j \in J$, and take an $N$ such that $j \in J_n$ for all $n \geq N$. Define
\[ \Phi_n=S_n^{-1} \g^* \g f - S^{-1} \g^* \g f, \quad n\geq N. \]
Then 
\begin{eqnarray*}
 S\Phi_n &= & S S_n^{-1} \g^* \g f - \g^* \g f \\
 & = & S_n S_n^{-1} \g^* \g f + \sum_{j \in J \setminus J_n} \g^* \g S_n^{-1} \g^* \g f - \g^* \g f \\
 &=&  \sum_{j \in J \setminus J_n} \g^* \g S_n^{-1} \g^* \g f,
\end{eqnarray*}
thus \[ \Phi_n=\sum_{j \in J \setminus J_n} S^{-1} \g^* \g S_n^{-1} \g^* \g f. \]
Therefore, for $g \in \BH$, we obtain
\begin{eqnarray*}
|\langle g, \Phi_n \rangle |^2 &=& \bigg| \bigg\langle g, \sum_{j \in J \setminus J_n} S^{-1} \g^* \g S_n^{-1} \g^* \g f \bigg\rangle \bigg|^2 \\
&=& \bigg| \sum_{j \in J \setminus J_n} \langle \g S^{-1} g, \g S_n^{-1} \g^* \g f \rangle \bigg|^2 \\
& \leq & \sum_{j \in J \setminus J_n} \| \g S^{-1} g \|^2 \sum_{j \in J \setminus J_n} \| \g S_n^{-1} \g^* \g f \|^2 \\
& \leq & B \| S_n^{-1} \g^* \g f \|^2 \sum_{j \in J \setminus J_n} \| \g S^{-1} g \|^2 \\
& \leq & B c_j^2 \sum_{j \in J \setminus J_n} \| \g S^{-1} g \|^2 \to 0 \quad \mathrm{as} \quad n \to \infty.
\end{eqnarray*}
Hence $|\langle g, \Phi_n \rangle | \to 0$ as $n \to \infty$, i.e., $\langle g, S_n^{-1} \g^* \g f \rangle \to \langle g, S^{-1} \g^* \g f \rangle$ as $n \to \infty$.
\end{proof}
The orthogonal projection $P_n :\BH \to \BH_n$ is given by $P_nf=\sum_{j \in J_n} S_n^{-1} \g^* \g f$ for all $f \in \BH.$ Since $\{P_n\}_{n=1}^\infty$ is increasing and 
$\overline{(\cup_{n=1}^\infty \BH_n})=\BH$, we have $P_nf \to f=\sum_{j \in J} S^{-1} \g^* \g f$ as $n \to \infty.$ Following Christensen \cite{chr96}, we say that the $projection \; method$ works if (\ref{eq11}) is satisfied for every $f,g \in \BH$ and the $strong \; projection \; method$ works if  
\[ \sum_{j \in J_n}| \langle f, S_n^{-1} \g^* \g f-S^{-1} \g^* \g f  \rangle |^2 \to 0 \quad \mathrm{as} \quad n \to \infty,\]
is satisfied for every $f \in \BH$. Note that the projection method works if the strong projection
method works. Since for any $f \in \BH$, we have
\begin{eqnarray*}
\sum_{j \in J_n}| \langle f, S_n^{-1} \g^* \g f-S^{-1} \g^* \g f  \rangle |^2 & = & \sum_{j \in J_n}| \langle P_n f, S_n^{-1} \g^* \g f \rangle - \langle f, S^{-1} \g^* \g f  \rangle |^2 \\
&=& \sum_{j \in J_n}| \langle \g (S_n^{-1} P_n f - S^{-1} f), \g f \rangle |^2 \\
& \leq & \sum_{j \in J_n} \| \g (S_n^{-1} P_n f - S^{-1} f) \|^2 \cdot \sum_{j \in J_n} \| \g f \|^2 \\
& \leq & B^2 \| S_n^{-1} P_n f - S^{-1} f \|^2 \cdot \| f \|^2,
\end{eqnarray*}
it follows that the strong projection method works if any one of the conditions appearing in Theorem \ref{th4} is satisfied. The result stated in the following can be found in (\cite{chr96}, Theorem 4.5) for Hilbert space frames. We generalize that result to HS-frames as follows.

\begin{thm}\label{th4}
Let $\{\g: j \in J \}$ be a  HS-frame for $\BH$ with the upper bound $B$. Then the following are equivalent:
\begin{enumerate}
\item $\| S_n^{-1} P_n f -S^{-1}f \| \to 0 \quad $ as $ n \to \infty, \; \forall f \in \BH.$
\item $\| (S-S_n)S_n^{-1} P_n f \| \to 0 \quad $ as $ n \to \infty, \; \forall f \in \BH.$
\item $\sum_{j \in J \setminus J_n} \|\g S_n^{-1} P_n f \|^2 \to 0 \quad$ as $ n \to \infty, \; \forall f \in \BH.$
\end{enumerate}
\end{thm}
\begin{proof}
$(1) \Leftrightarrow (2)$ Let $f \in \BH$. Then we have
\begin{eqnarray*}
S_n^{-1} P_n f -S^{-1}f &=& S^{-1}(P_nf-f)+S^{-1}(S-S_n)S_n^{-1} P_n f \\
(S-S_n)S_n^{-1} P_n f &=& S(S_n^{-1} P_n f -S^{-1}f)-(P_nf-f) \\
\Rightarrow \quad \quad \| S_n^{-1} P_n f -S^{-1}f \| & \leq & \| S^{-1} \| \cdot \|P_nf-f \|+ \|S^{-1} \| \cdot \|(S-S_n)S_n^{-1} P_n f \| \\
\| (S-S_n)S_n^{-1} P_n f \| & \leq & \| S \| \cdot \|S_n^{-1} P_n f -S^{-1}f\| + \|P_nf-f\|.
\end{eqnarray*}
Since $\|P_nf -f\| \to 0$ as $n \to \infty$, we obtain that $(1)$ and $(2)$ are equivalent.\\
$(1) \Rightarrow (3)$ For every $f \in \BH$, we have
\begin{eqnarray*}
\bigg( \sum_{j \in J \setminus J_n} \|\g S_n^{-1} P_n f \|^2 \bigg)^{\frac{1}{2}} &\leq & \bigg( \sum_{j \in J \setminus J_n} \|\g( S_n^{-1} P_n f -S^{-1}f ) \|^2 \bigg)^{\frac{1}{2}} + \bigg( \sum_{j \in J \setminus J_n} \|\g S^{-1}f \|^2 \bigg)^{\frac{1}{2}} \\
&\leq & \sqrt{B} \| S_n^{-1} P_n f -S^{-1}f \| + \bigg( \sum_{j \in J \setminus J_n} \|\g S^{-1}f \|^2 \bigg)^{\frac{1}{2}}.
\end{eqnarray*}
Since $ \sum_{j \in J \setminus J_n} \|\g S^{-1}f \|^2 \to 0$ as $n \to \infty$, the result follows. \\
$(3) \Rightarrow (2)$ For every $f \in \BH$, we obtain 
\begin{eqnarray*}
\| (S-S_n)S_n^{-1} P_n f \|^2 & =& \sup_{\| g \|=1} | \langle (S-S_n)S_n^{-1} P_n f , g \rangle|^2 \\ 
& =& \sup_{\| g \|=1} \bigg| \bigg\langle \sum_{j \in J \setminus J_n} \g^* \g S_n^{-1} P_n f , g  \bigg\rangle \bigg|^2 \\
& \leq & \sup_{\| g \|=1} \sum_{j \in J \setminus J_n} \| \g S_n^{-1} P_n f \|^2 \cdot \sum_{j \in J \setminus J_n} \| \g  g \|^2 \\
& \leq & B \sum_{j \in J \setminus J_n} \| \g S_n^{-1} P_n f \|^2.
\end{eqnarray*}
Since $\sum_{j \in J \setminus J_n} \| \g S_n^{-1} P_n f \|^2 \to 0$ as $n \to \infty$, we have the desired result.
\end{proof}
Now we derive a general method for approximation of the inverse HS-frame operator. We first establish the following result, which generalizes Lemmas 3.1 and 3.2 in \cite{cas2000} to HS-frames in a more general form.

\begin{prop}\label{pro3}
Let $\{\g: j \in J \}$ be a HS-frame for $\BH$ with bounds $A$ and $B$. Let $\lambda > 1$ be a scalar. Then for any $n \in \BN$ there exists a number $m(n)$ such that the following holds:
\begin{enumerate}
\item $\frac{A}{\lambda} \| f \|^2 \leq \sum_{j \in J_{n+m(n)}} \| \g (f)\|^2$ for all $f \in \BH_n.$ 
\item $\{\g P_n\}_{j \in J_{n+m(n)}}$ is a HS-frame for $\BH_n$ with bounds $A/ \lambda$ and $B$. Moreover, the HS-frame operator for $\{\g P_n\}_{j \in J_{n+m(n)}}$ is $P_n S_{n+m(n)}: \BH_n \to \BH_n$, with 
\[ \| P_n S_{n+m(n)} \| \leq B, \quad \mathrm{and} \quad \| (P_n S_{n+m(n)})^{-1} \| \leq  \frac{\lambda}{A}. \]
\end{enumerate}
\end{prop}
\begin{proof} (1)
Let $n \in \BN$ and $\lambda > \mu >1$. Choose $\varepsilon > 0$ such that $\sqrt{A/\mu} -\sqrt{B}\varepsilon \geq \sqrt{A/\lambda}$. Since $\{ f \in \BH_n: \| f \|=1\}$ is compact, there exist a finite set of elements $g_k \in \BH_n$ with $\| g_k \|=1,$ for all $k$ such that the balls $B(g_k, \varepsilon)=\{f \in \BH_n: \|f-g_k\| \leq \varepsilon\}$ cover the set $\{ f \in \BH_n: \| f \|=1\}$. Since $\{\g: j \in J \}$ is a HS-frame for $\BH$, we have $A \leq \sum_{j \in J} \| \g (g_k) \|^2$ for all $k$. Hence we can choose $m(n)$ such that 
\[ \frac{A}{\mu} \leq \sum_{j \in J_{n+m(n)}} \| \g (g_k)\|^2, \quad \forall k. \]
Now let $f \in \BH_n$ with $\| f \|=1$. Choose $k$ such that $f \in B(g_k, \varepsilon)$. Therefore 
\begin{eqnarray*}
\bigg( \sum_{j \in J_{n+m(n)}} \| \g (f)\|^2 \bigg)^{\frac{1}{2}} & \geq & \bigg( \sum_{j \in J_{n+m(n)}} \| \g (g_k)\|^2 \bigg)^{\frac{1}{2}} -\bigg( \sum_{j \in J_{n+m(n)}} \| \g (f-g_k)\|^2 \bigg)^{\frac{1}{2}} \\
& \geq & \sqrt{A/ \mu} - \sqrt{B} \| f- g_k \| \geq \sqrt{A / \mu} - \sqrt{B} \varepsilon \geq \sqrt{A / \lambda}.
\end{eqnarray*} 
(2) Since $P_n f=f$ for all $f \in \BH_n$, from (1) we get
\[ \frac{A}{\lambda} \| f \|^2 \leq \sum_{j \in J_{n+m(n)}} \| \g (f)\|^2 = \sum_{j \in J_{n+m(n)}} \| \g P_n(f)\|^2 \leq \sum_{j \in J} \| \g P_n(f)\|^2 \leq B \| f \|^2, \; \forall f \in \BH_n. \]
Hence $\{\g P_n\}_{j \in J_{n+m(n)}}$ is a HS-frame for $\BH_n$ with bounds $A/ \lambda$ and $B$. Moreover, 
\[ P_n S_{n+m(n)} f= \sum_{j \in J_{n+m(n)}} P_n \g^* \g P_n f = \sum_{j \in J_{n+m(n)}} ( \g P_n)^* (\g P_n) f, \;\; \forall f \in \BH_n. \]
Therefore $P_n S_{n+m(n)}$ is the HS-frame operator for $\{\g P_n\}_{j \in J_{n+m(n)}}$. Now the norm estimates follow from the fact that
\[ B = \sup_{\| f \|=1} \sum_{j \in J} \| \g P_n(f)\|^2 =  \sup_{\| f \|=1} \langle P_n S_{n+m(n)}f,f \rangle = \| P_n S_{n+m(n)} \|,\]
and $\| (P_n S_{n+m(n)})^{-1} \| =  \lambda /A$, which follows from the properties of dual HS-frames.
\end{proof}

\begin{rmk}
If we consider $\lambda=2$ in Proposition \ref{pro3}, then we obtain the similar inequalities as in Lemmas 3.1 and 3.2 in \cite{cas2000}.
\end{rmk}
Now we are ready to prove that $S^{-1}$ can be approximated arbitrarily closely in the strong operator topology using the operators $(P_n S_{n+m(n)})^{-1} P_n$. A similar result for Hilbert space frames can be found in (\cite{cas2000}, Theorem 3.3). We use Proposition \ref{pro3} and give a similar proof for HS-frames as follows.

\begin{thm}
Let $\{\g: j \in J \}$ be a HS-frame for $\BH$ with bounds $A$ and $B$. For a fix $\lambda > 1$, and for any $n \in \BN$, choose $m(n)$ such that for all $f \in \BH_n$ \[ \frac{A}{\lambda}\| f \|^2 \leq \sum_{j \in J_{n+m(n)}} \| \g (f) \|^2. \]
Then $(P_n S_{n+m(n)})^{-1} P_n f \to S^{-1} f$ as $n \to \infty$, for all $f \in \BH$.
\end{thm}
\begin{proof}
Let $f \in \BH$. Since $(P_n -I) S^{-1} f \to 0$ as $n \to \infty$ and
\[(P_n S_{n+m(n)})^{-1} P_n f - S^{-1} f = (P_n S_{n+m(n)})^{-1} P_n f -P_n S^{-1} f + (P_n -I) S^{-1} f, \]
it is enough to show that $(P_n S_{n+m(n)})^{-1} P_n f - P_n S^{-1} f \to 0 $ as $n \to \infty.$ Using Proposition \ref{pro3}, we obtain
\begin{eqnarray*}
&& \| (P_n S_{n+m(n)})^{-1} P_n f - P_n S^{-1} f \| \\ 
& \leq & \| (P_n S_{n+m(n)})^{-1} \| \cdot \| P_n f - P_n S_{n+m(n)}P_n S^{-1} f \| \\
& \leq & \frac{\lambda}{A} \| S_{n+m(n)}P_n S^{-1} f -f \|  \\
& \leq & \frac{\lambda}{A} \bigg( \| S_{n+m(n)}(P_n-I) S^{-1} f \| + \| S_{n+m(n)} S^{-1} f -f \| \bigg) \\
& \leq & \frac{\lambda}{A} \bigg( B \| (P_n-I) S^{-1} f \| + \bigg\| \sum_{j \in J \setminus J_{n+m(n)}} \g^* \g S^{-1} f \bigg\| \bigg) \to 0 \quad \mathrm{as} \quad n \to \infty. 
\end{eqnarray*}
Hence, we have the desired result.
\end{proof}
Finally we generalize Theorem 4 in \cite{cass97} from the setting of Hilbert space frames to HS-frames and we include a similar proof.

\begin{thm}
Let $\{\g: j \in J \}$ be a HS-frame for $\BH$.  Then the following are equivalent:
\begin{enumerate}
\item $\sum_{j \in  J_n}S_n^{-1} \g^*(\a) \to \sum_{j \in  J}S^{-1} \g^*(\a)$ as $n \to \infty$, for all $\{\a \}_{j \in J} \in \bigoplus C_2$.
\item $S_n^{-1}\sum_{j \in J_n}\g^*(\a) \to 0$ as $n \to \infty$ for all  $\{\a \}_{j \in J} \in \bigoplus C_2$ with $\sum_{j \in J}\g^*(\a)=0$.
\end{enumerate}
\end{thm}
\begin{proof} 
Let $T$ be the synthesis operator for $\{\g: j \in J \}$. Since $\bigoplus C_2$ is the orthogonal sum of the range of $T^*$ and the kernel of $T$, we can write any  $\{\a \}_{j \in J} \in \bigoplus C_2$ as $\{\a \}_{j \in J}=\{\g(g) \}_{j \in J} + \{\mathcal{F}_j \}_{j \in J}$ for some $g \in \BH$ and $\{\mathcal{F}_j \}_{j \in J} \in \mathrm{Ker \;} T .$ Then
\[\sum_{j \in  J_n}S_n^{-1} \g^*(\a)=\sum_{j \in  J_n}S_n^{-1} \g^* \g(g)+\sum_{j \in  J_n}S_n^{-1} \g^*(\mathcal{F}_j) = P_n g+ S_n^{-1} \sum_{j \in  J_n} \g^*(\mathcal{F}_j). \]
Also, we have
\[ \sum_{j \in  J}S^{-1} \g^*(\a)=\sum_{j \in  J}S^{-1} \g^* \g (g)+\sum_{j \in  J}S^{-1} \g^*(\mathcal{F}_j) =g+ S^{-1} \sum_{j \in  J} \g^*(\mathcal{F}_j)=g, \]
from which (1) and (2) are equivalent.
\end{proof}

\section{Stability of HS-frames}\label{sec4}
In this section, we study the stability of HS-frames. Before we prove the main results of this section, we first need the following lemma.
\begin{lem}\label{lem3}
\cite{cas97} Let $\mathcal{X}$ be a Banach space, $U: \mathcal{X} \rightarrow \mathcal{X}$ is a linear operator. If there exist constants $\lambda_1, \lambda_2 \in [0,1)$ such that \[ \Vert Ux -x \Vert \leq \lambda_1 \Vert x \Vert+\lambda_2 \Vert Ux \Vert, \; \forall x \in \mathcal{X}. \]
Then $U$ is a bounded invertible operator on $\mathcal{X}$, and 
\[ \frac{1-\lambda_1}{1+\lambda_2} \Vert x \Vert \leq \Vert Ux \Vert \leq \frac{1+\lambda_1}{1-\lambda_2} \Vert x \Vert, \;\frac{1-\lambda_2}{1+\lambda_1} \Vert x \Vert \leq \Vert U^{-1}x \Vert \leq \frac{1+\lambda_2}{1-\lambda_1} \Vert x \Vert, \; \forall x \in \mathcal{X}. \]
\end{lem}

The following is a fundamental result in the study of the stability of frames.
\begin{prop}
$($\cite{cas97}, $\mathrm{Theorem \;2})$
Let $\{f_i \}_{i=1}^\infty$ be a frame for some Hilbert space $\BH$ with bounds $A, B$. Let $\{g_i \}_{i=1}^\infty \subseteq \BH$ and assume that there exist constants $\lambda_1, \lambda_2, \mu \geq 0$ such that $\max(\lambda_1+\frac{\mu}{\sqrt{A}}, \lambda_2)<1$ and
\begin{equation}\label{eq1}
\bigg\Vert \sum_{i=1}^n c_i (f_i-g_i) \bigg\Vert \leq \lambda_1 \bigg\Vert \sum_{i=1}^n c_i f_i \bigg\Vert +\lambda_2 \bigg\Vert \sum_{i=1}^n c_i g_i \bigg\Vert + \mu \bigg[ \sum_{i=1}^n |c_i|^2 \bigg]^{1/2} 
\end{equation} 
for all $c_1,...,c_n(n \in \BN).$ Then $\{g_i \}_{i=1}^\infty$ is a frame for $\BH$ with bounds
\[ A \left( 1 - \frac{\lambda_1+\lambda_2+\frac{\mu}{\sqrt{A}}}{1+\lambda_2} \right)^2, \;\; B \left( 1+ \frac{\lambda_1+\lambda_2+\frac{\mu}{\sqrt{B}}}{1-\lambda_2} \right)^2.\]
\end{prop}
Similar to ordinary frames, HS-frames are stable under small perturbations. The stability of HS-frames is discussed in the following theorem. 
\begin{thm}\label{th1}
Let $\{\g : j \in J\}$ be a HS-frame for $\BH$ with respect to $\BK$. Let $A, B$ be the frame bounds. Suppose that $\ga \in \mathcal{L}(\BH, C_2)$ and there exist constants $\lambda_1, \lambda_2, \mu \geq 0$ such that $\max(\lambda_1+\frac{\mu}{\sqrt{A}}, \lambda_2)<1$ and one of the following two conditions is satisfied:
\begin{eqnarray}\label{eq2}
&& \bigg( \sum_{j \in J} \Vert (\g -\ga)f \Vert^2 \bigg)^{1/2} \nonumber \\ &\leq & \lambda_1 \bigg(\sum_{j \in J} \Vert \g (f) \Vert^2 \bigg)^{1/2}+\lambda_2 \bigg(\sum_{j \in J} \Vert \ga (f) \Vert^2 \bigg)^{1/2} +\mu \Vert f \Vert, \; \forall f \in \BH,
\end{eqnarray}
or
\begin{eqnarray}\label{eq3}
&& \bigg\Vert \sum_{j \in J_1} (\g^* -\ga^*)\a \bigg\Vert \nonumber \\ & \leq & \lambda_1 \bigg\Vert \sum_{j \in J_1} \g^* (\a) \bigg\Vert + \lambda_2 \bigg\Vert \sum_{j \in J_1} \ga^* (\a) \bigg\Vert+ \mu \bigg( \sum_{j \in J_1} \Vert \a \Vert^2 \bigg)^{1/2},
\end{eqnarray}
for any finite subset $J_1 \subset J$ and $\a \in C_2$. Then $\{\ga: j \in J\}$ is a HS-frame for $\BH$ with bounds
\begin{equation}\label{eq10}
A \left( 1 - \frac{\lambda_1+\lambda_2+\frac{\mu}{\sqrt{A}}}{1+\lambda_2} \right)^2, \;\; B \left( 1+ \frac{\lambda_1+\lambda_2+\frac{\mu}{\sqrt{B}}}{1-\lambda_2} \right)^2.
\end{equation}
\end{thm}
\begin{proof}
First, we assume that (\ref{eq2}) is satisfied. Notice that 
\[ \sum_{j \in J} \Vert \g (f) \Vert^2 \leq B \Vert f \Vert^2. \]
From (\ref{eq2}) we see that
\[ \bigg( \sum_{j \in J} \Vert (\g -\ga)f \Vert^2 \bigg)^{1/2} \leq \bigg(\lambda_1 \sqrt{B} +\mu \bigg)\Vert f \Vert +\lambda_2 \bigg(\sum_{j \in J} \Vert \ga (f) \Vert^2 \bigg)^{1/2}. \]
Using the triangle inequality, we get
\[ \bigg( \sum_{j \in J} \Vert (\g -\ga)f \Vert^2 \bigg)^{1/2} \geq \bigg(\sum_{j \in J} \Vert \ga (f) \Vert^2 \bigg)^{1/2} - \bigg(\sum_{j \in J} \Vert \g (f) \Vert^2 \bigg)^{1/2}. \] 
Hence
\begin{eqnarray*}
(1-\lambda_2) \bigg(\sum_{j \in J} \Vert \ga (f) \Vert^2 \bigg)^{1/2} & \leq & \left(\lambda_1 \sqrt{B} +\mu \right)\Vert f \Vert + \bigg(\sum_{j \in J} \Vert \g (f) \Vert^2 \bigg)^{1/2} \\
 & \leq & \sqrt{B} \bigg( 1 + \lambda_1 + \frac{\mu}{\sqrt{B}} \bigg)\Vert f \Vert. 
\end{eqnarray*}
Therefore,
\[ \sum_{j \in J} \Vert \ga (f) \Vert^2 \leq B \left( 1+ \frac{\lambda_1+\lambda_2+\frac{\mu}{\sqrt{B}}}{1-\lambda_2} \right)^2 \Vert f \Vert^2. \]
Similarly we can prove that
\[ \sum_{j \in J} \Vert \ga (f) \Vert^2 \geq A \left( 1 - \frac{\lambda_1+\lambda_2+\frac{\mu}{\sqrt{A}}}{1+\lambda_2} \right)^2 \Vert f \Vert^2. \]
Next, we assume that (\ref{eq3}) is satisfied. Let $T$ and $S$ denote the synthesis operator and frame operator associated with $\{\g : j \in J\}$. Also, let $V$ denote the synthesis operator associated with $\{\ga : j \in J\}$. Since $\{\g : j \in J\}$ is a HS-frame for $\BH$ with bounds $A$ and $B$, by Proposition \ref{pro1}, $T$ is a bounded operator with $\Vert T \Vert \leq \sqrt{B}.$ From the inequality (\ref{eq3}), using the triangle inequality, we get
\[ \bigg\Vert \sum_{j \in J_1} \ga^* (\a) \bigg\Vert \leq \frac{1+\lambda_1}{1-\lambda_2}  \bigg\Vert \sum_{j \in J_1} \g^* (\a) \bigg\Vert + \frac{\mu}{1-\lambda_2} \bigg( \sum_{j \in J_1} \Vert \a \Vert^2 \bigg)^{1/2}. \] So for any $\{\a: j \in J\} \in \bigoplus C_2$, the series $\sum_{j \in J} \ga^* (\a)$ is convergent. Hence $\{\ga : j \in J\}$ is a HS-Bessel sequence for $\BH$. Using the definition of a synthesis operator, we get
\begin{eqnarray*}
\Vert V(\{\a\}_{j \in J}) \Vert & \leq & \frac{1+\lambda_1}{1-\lambda_2} \Vert T(\{\a\}_{j \in J}) \Vert + \frac{\mu}{1-\lambda_2} \Vert \{\a\}_{j \in J} \Vert \\
& \leq & \frac{(1+\lambda_1) \sqrt{B}+\mu}{1-\lambda_2} \Vert \{\a\}_{j \in J} \Vert \\
& =&  \sqrt{B} \bigg(1+ \frac{\lambda_1+\lambda_2+\frac{\mu}{\sqrt{B}}}{1-\lambda_2}\bigg) \Vert \{\a\}_{j \in J} \Vert, \; \forall \{\a\}_{j \in J} \in \bigoplus C_2.
\end{eqnarray*}
It implies that $\{\ga : j \in J\}$ is a HS-Bessel sequence with bound $B (1+ \frac{\lambda_1+\lambda_2+\mu/ \sqrt{B}}{1-\lambda_2})^2$. For any $f \in \BH$, let $\{\a\}_{j \in J}=\{\g S^{-1}f\}_{j \in J} \in \bigoplus C_2$. Using the inequality (\ref{eq3}) on the sequence $\{\g S^{-1}f\}_{j \in J}$ we obtain that
\begin{eqnarray*}
&& \bigg\Vert \sum_{j \in J} (\g^* -\ga^*)\g S^{-1}f \bigg\Vert \nonumber \\ & \leq & \lambda_1 \bigg\Vert \sum_{j \in J} \g^* \g S^{-1}f \bigg\Vert + \lambda_2 \bigg\Vert \sum_{j \in J} \ga^* \g S^{-1}f \bigg\Vert+ \mu \bigg( \sum_{j \in J} \Vert \g S^{-1}f \Vert^2 \bigg)^{1/2}.
\end{eqnarray*}
Since for any $ f \in \BH,$ we have
\[ \sum_{j \in J} \g^* \g S^{-1}f=f, \sum_{j \in J} \ga^* \g S^{-1}f=VT^*S^{-1}f \mathrm{\;and\;} \bigg( \sum_{j \in J} \Vert \g S^{-1}f \Vert^2 \bigg)^{1/2} \leq \frac{1}{\sqrt{A}} \Vert f \Vert,\] 
from the above inequality we obtain
\[\Vert f-  VT^*S^{-1}f \Vert \leq \bigg( \lambda_1+\frac{\mu}{\sqrt{A}} \bigg) \Vert f \Vert + \lambda_2 \Vert VT^*S^{-1}f \Vert, \forall f \in \BH. \]
So, by Lemma \ref{lem3}, the operator $VT^*S^{-1}$ is invertible, and 
\[ \Vert VT^*S^{-1} \Vert \leq \frac{1+\lambda_1+\frac{\mu}{\sqrt{A}}}{1-\lambda_2}, \;  \Vert ( VT^*S^{-1})^{-1} \Vert \leq \frac{1+\lambda_2}{1-(\lambda_1+\frac{\mu}{\sqrt{A}})}. \]
Every $f \in \BH$ can be written as
\[ f= VT^*S^{-1}(VT^*S^{-1})^{-1}f=\sum_{j \in J} \ga^* \g S^{-1}(VT^*S^{-1})^{-1}f. \]
It implies that
\begin{eqnarray*}
\langle f,f \rangle &=& \bigg\langle \sum_{j \in J} \ga^* \g S^{-1}(VT^*S^{-1})^{-1}f,f \bigg\rangle = \sum_{j \in J} \big[ \g S^{-1}(VT^*S^{-1})^{-1}f, \ga f \big]_{\tau} \\
& \leq & \sum_{j \in J} \Vert \g S^{-1}(VT^*S^{-1})^{-1}f \Vert \cdot \Vert \ga f \Vert \\ 
& \leq & \bigg( \sum_{j \in J} \Vert \g S^{-1}(VT^*S^{-1})^{-1}f \Vert^2 \bigg)^{1/2} \cdot \bigg( \sum_{j \in J} \Vert \ga f \Vert^2 \bigg)^{1/2} \\
& \leq & \frac{1}{\sqrt{A}} \Vert (VT^*S^{-1})^{-1}f \Vert \cdot \bigg( \sum_{j \in J} \Vert \ga f \Vert^2 \bigg)^{1/2} \\
& \leq & \frac{1}{\sqrt{A}} \bigg( \frac{1+\lambda_2}{1-(\lambda_1+\frac{\mu}{\sqrt{A}})} \bigg) \Vert f \Vert \cdot \bigg( \sum_{j \in J} \Vert \ga f \Vert^2 \bigg)^{1/2}, \forall f \in \BH.
\end{eqnarray*} 
Therefore
\begin{eqnarray*}
 \sum_{j \in J} \Vert \ga f \Vert^2 \geq A \bigg( \frac{1-(\lambda_1+\frac{\mu}{\sqrt{A}})}{1+\lambda_2} \bigg)^2 \Vert f \Vert^2= A \bigg(1- \frac{\lambda_1+\lambda_2+\frac{\mu}{\sqrt{A}}}{1+\lambda_2} \bigg)^2 \Vert f \Vert^2, \; \forall f \in \BH.
\end{eqnarray*}
This completes the proof.
\end{proof}
\begin{rmk}
In general, the inequality (\ref{eq2}) does not imply that $\{\ga : j \in J \}$ is a HS-frame regardless how small the parameters $\lambda_1, \lambda_2, \mu$ are. A counterexample for $g$-frames can be found in \cite{sun07}, and an example can be constructed similarly for HS-frames. 
\end{rmk}
\begin{cor}
Let $\{\g : j \in J\}$ be a HS-Riesz basis for $\BH$ with bounds $A$ and $B$. Assume that the condition (\ref{eq3}) in Theorem \ref{th1} is satisfied, then $\{\ga : j \in J\}$ is also a HS-Riesz basis for $\BH$ with bounds given by (\ref{eq10}).
\end{cor} 
\begin{proof}
From Theorem \ref{th3} and Theorem \ref{th1}, we obtain that $\{\ga : j \in J\}$ is a HS-frame for $\BH$ with bounds given by (\ref{eq10}). Let $\sum_{j \in J} \ga^*(\a)=0$ for $\{ \a \}_{j \in J} \in \bigoplus C_2$. Since $\{\g : j \in J\}$ is a HS-Riesz basis for $\BH$ with bounds $A$ and $B$, from (\ref{eq3}), we get
\[ \sqrt{A} \bigg( \sum_{j \in J} \Vert \a \Vert^2  \bigg)^{1/2} \leq \bigg\Vert \sum_{j \in J} \g^*(\a) \bigg\Vert \leq \frac{\mu}{1-\lambda_1} \bigg( \sum_{j \in J} \Vert \a \Vert^2  \bigg)^{1/2}. \]
It implies that \[ \bigg( 1 -\lambda_1 -\frac{\mu}{\sqrt{A}} \bigg) \bigg( \sum_{j \in J} \Vert \a \Vert^2  \bigg)^{1/2} \leq 0. \]
Since $1 - \lambda_1 - \frac{\mu}{\sqrt{A}} > 0$, $\sum_{j \in J} \Vert \a \Vert^2 =0$. Hence $\a=0$ for all $j \in J$. It follows that $\{\ga : j \in J\}$ is an $\bigoplus C_2$-linearly independent family. From Theorem \ref{th3}, we find that $\{\ga : j \in J\}$ is a HS-Riesz basis for $\BH$ with bounds given by (\ref{eq10}), which completes the proof.
\end{proof}

\begin{cor}
Let $\{\g : j \in J\}$ be a HS-frame for $\BH$ with bounds $A, B$, and let $\{\ga: j \in J\}$ be a sequence in $\mathcal{L}(\BH, C_2)$. Assume that there exists a constant $0<M<A$ such that 
\[ \sum_{j \in J} \Vert (\g -\ga)f \Vert^2 \leq M \Vert f \Vert^2, \; \forall f \in \BH, \] then $\{\ga : j \in J\}$ is a HS-frame for $\BH$ with bounds $A[1-(M/A)^{1/2}]^2$ and $B[1+(M/B)^{1/2}]^2$.
\end{cor}
\begin{proof}
Let $\lambda_1=\lambda_2=0$ and $\mu=\sqrt{M}.$ Since $M< A$, $\mu/\sqrt{A}=\sqrt{M/A}<1.$ So, by Theorem \ref{th1}, $\{\ga : j \in J\}$ is a HS-frame for $\BH$ with bounds $A[1-(M/A)^{1/2}]^2$ and $B[1+(M/B)^{1/2}]^2$. 
\end{proof}
In \cite{guo13}, the author established the various perturbation results on $g$-frames in Hilbert spaces. Motivated by his results, in the following, we discuss some interesting perturbation results for HS-frames.
\begin{thm}\label{th2}
Let $\{\g : j \in J\}$ be a HS-frame for $\BH$ with bounds $A, B$ and $\{\ga: j \in J\} \subseteq \mathcal{L}(\BH, C_2)$ be a HS-Bessel sequence with bound $D$. Assume that there exist constants $\lambda_1, \lambda_2, \mu, \nu \geq 0$ such that $\max\{\lambda_1+\frac{\mu}{\sqrt{A}}+ \frac{\nu}{A} \cdot \sqrt{D}, \lambda_2 \}< 1$ and the following condition is satisfied,
\begin{eqnarray}\label{eq5}
\bigg\Vert \sum_{j \in J} (\g^* \g f -\ga^* \ga f) \bigg\Vert  & \leq &  \lambda_1 \bigg\Vert \sum_{j \in J} \g^*\g f  \bigg\Vert + \lambda_2 \bigg\Vert \sum_{j \in J} \ga^* \ga f \bigg\Vert \nonumber \\ 
&& + \mu \bigg( \sum_{j \in J} \Vert \g f \Vert^2 \bigg)^{1/2}+ \nu \bigg( \sum_{j \in J} \Vert \ga f \Vert^2 \bigg)^{1/2}, \; \forall f \in \BH.
\end{eqnarray}
Then $\{\ga: j \in J\}$ is a HS-frame for $\BH$.
\end{thm}
\begin{proof}
Let $Sf=\sum_{j \in J}\g^* \g f$ and $Gf=\sum_{j \in J}\ga^* \ga f$. Since $\{\g : j \in J\}$ is a HS-frame and $\{\ga: j \in J\}$ is a HS-Bessel sequence, $S$ is invertible and $G$ is a
bounded operator on $\BH$. From the inequality (\ref{eq5}), for each $f \in \BH$ we have 
\begin{eqnarray*}
\Vert Sf -Gf \Vert & \leq & \lambda_1 \Vert Sf \Vert +\lambda_2 \Vert Gf \Vert+\mu \bigg( \sum_{j \in J} \Vert \g f \Vert^2 \bigg)^{1/2}+ \nu \bigg( \sum_{j \in J} \Vert \ga f \Vert^2 \bigg)^{1/2} \\
& \leq & \lambda_1 \Vert Sf \Vert +\lambda_2 \Vert Gf \Vert+\mu  \bigg( \sum_{j \in J} \Vert \g f \Vert^2 \bigg)^{1/2} + \nu \cdot \sqrt{D} \Vert f \Vert.
\end{eqnarray*}
Therefore 
\begin{eqnarray*}
\Vert f -GS^{-1}f \Vert & \leq & \lambda_1 \Vert f \Vert +\lambda_2 \Vert GS^{-1}f \Vert + \mu  \bigg( \sum_{j \in J} \Vert \g S^{-1}f \Vert^2 \bigg)^{1/2} + \nu  \sqrt{D} \Vert S^{-1}f \Vert \\
 & \leq & \lambda_1 \Vert f \Vert + \bigg( \frac{\mu}{\sqrt{A}} +\frac{\nu}{A} \cdot \sqrt{D} \bigg)  \Vert f  \Vert +\lambda_2 \Vert GS^{-1}f \Vert \\
 & = & \bigg( \lambda_1+ \frac{\mu}{\sqrt{A}} +\frac{\nu}{A} \cdot \sqrt{D} \bigg)  \Vert f  \Vert +\lambda_2 \Vert GS^{-1}f \Vert.
\end{eqnarray*}
Since $\max\{\lambda_1+\frac{\mu}{\sqrt{A}} + \frac{\nu}{A} \cdot \sqrt{D}, \lambda_2 \}< 1$, by Lemma \ref{lem3}, $GS^{-1}$ is invertible and consequently $G$ is invertible. It follows that $\{\ga: j \in J\}$ is a HS-frame for $\BH$.
\end{proof}
\begin{cor}
Let $\{\g : j \in J\}$ be a HS-frame for $\BH$ with bounds $A, B$ and $\{\ga: j \in J\} \subseteq \mathcal{L}(\BH, C_2)$ be a family of operators. Assume that there exist a constant $0<M<A$ such that
\[  \sum_{j \in J} \Vert \g^* \g f -\ga^* \ga f \Vert \leq M \Vert f \Vert, \; \forall f \in \BH, \]
then $\{\ga : j \in J\}$ is a HS-frame for $\BH$.
\end{cor}
\begin{proof}
For each $f \in \BH$, we have
\begin{eqnarray*}
\bigg\Vert \sum_{j \in J} \ga^* \ga f \bigg\Vert \leq \bigg\Vert \sum_{j \in J}  ( \g^* \g f -\ga^* \ga f) \bigg\Vert + \bigg\Vert \sum_{j \in J} \g^* \g f \bigg\Vert \leq (M+B)\Vert f \Vert.
\end{eqnarray*}
Thus $\sum_{j \in J} \ga^* \ga f$ is convergent for each $f \in \BH.$ Therefore for all $f \in \BH$
\[ \sum_{j \in J} \Vert \ga f \Vert^2= \sum_{j \in J} \langle \ga^* \ga f, f \rangle = \bigg\langle \sum_{j \in J}  \ga^* \ga f, f \bigg\rangle \leq \bigg\Vert \sum_{j \in J} \ga^* \ga f \bigg\Vert \cdot \Vert f \Vert \leq (M+B)\Vert f \Vert^2. \]
It follows that $\{\ga : j \in J\}$ is a HS-Bessel sequence for $\BH$. Also we have
\[ \bigg\Vert \sum_{j \in J} (\g^* \g f -\ga^* \ga f) \bigg\Vert \leq M \Vert f \Vert \leq \frac{M}{\sqrt{A}} \bigg( \sum_{j \in J} \Vert \g (f) \Vert^2 \bigg)^{1/2}, \; f \in \BH. \]
Let $\lambda_1=\lambda_2=\nu=0$ and $\mu = M/\sqrt{A}$. Since $M<A$, $\mu/\sqrt{A}=M/A<1.$  So, by Theorem \ref{th2}, $\{\ga : j \in J\}$ is a HS-frame for $\BH$.
\end{proof}
\begin{thm}
Let $\{\g : j \in J\}$ be a HS-frame for $\BH$ with bounds $A, B$ and $\{\ga: j \in J\} \subseteq \mathcal{L}(\BH, C_2)$ be a HS-Bessel sequence for $\BH$. Assume that there exist constants $\lambda_1, \lambda_2, \mu \geq 0$ with $\max\{\lambda_1+\frac{\mu}{\sqrt{A}}, \lambda_2 \}< 1$ such that
\begin{eqnarray}\label{eq6}
&& \bigg\Vert \sum_{j \in J} (\g^* \a -\ga^* \a) \bigg\Vert \nonumber \\ & \leq & \lambda_1 \bigg\Vert \sum_{j \in J} \g^* (\a) \bigg\Vert + \lambda_2 \bigg\Vert \sum_{j \in J} \ga^* (\a) \bigg\Vert+ \mu \bigg( \sum_{j \in J} \Vert \a \Vert^2 \bigg)^{1/2},
\end{eqnarray}
where $\{\a \}_{j \in J} \in \bigoplus C_2$, then $\{\ga : j \in J\}$ is a HS-frame for $\BH$.
\end{thm}
\begin{proof}
Let $T$ and $S$ denote the synthesis operator and frame operator associated with $\{\g : j \in J\}$. Also, let $V$ denote the synthesis operator associated with $\{\ga : j \in J\}$. From the inequality (\ref{eq6}), we obtain
\[ \Vert T(\{ \a \}_{j \in J})- V(\{ \a \}_{j \in J}) \Vert \leq \lambda_1 \Vert T(\{ \a \}_{j \in J}) \Vert + \lambda_2 \Vert V(\{ \a \}_{j \in J}) \Vert+ \mu \bigg( \sum_{j \in J} \Vert \a \Vert^2 \bigg)^{1/2}. \]
For any $f \in \BH$, let $\{ \a \}_{j \in J}=\{\g S^{-1}f \}_{j \in J} \in \bigoplus C_2$, then
\begin{eqnarray*}
&& \Vert T(\{\g S^{-1}f \}_{j \in J})- V(\{\g S^{-1}f \}_{j \in J}) \Vert  =  \Vert T T^*(S^{-1}f)- V T^*(S^{-1}f) \Vert \\ 
& = & \Vert SS^{-1}f- V T^*S^{-1}f \Vert = \Vert f- V T^*S^{-1}f \Vert \\
& \leq & \lambda_1 \Vert f \Vert + \lambda_2 \Vert V T^*S^{-1}f \Vert + \frac{\mu}{\sqrt{A}} \Vert f \Vert = \bigg( \lambda_1+\frac{\mu}{\sqrt{A}} \bigg) \Vert f \Vert + \lambda_2 \Vert V T^*S^{-1}f \Vert.
\end{eqnarray*}
Since $\max\{\lambda_1+\frac{\mu}{\sqrt{A}}, \lambda_2 \}< 1$, by Lemma \ref{lem3}, the operator $V T^*S^{-1}$ is invertible and hence $V$ is surjective. Then by Proposition \ref{pro2}, the sequence $\{\ga : j \in J\}$ is a HS-frame for $\BH$.
\end{proof}
\begin{rmk}
Since $g$-frames can be considered as a class of HS-frames, the previous results on $g$-frames can be obtained as a special case of the results we established for HS-frames.
\end{rmk}

\section*{Acknowledgments}
The author is deeply indebted to Prof. Radu Balan for several valuable comments and suggestions. The author is grateful to the United States-India Educational Foundation for providing the Fulbright-Nehru Doctoral Research Fellowship, and Department of Mathematics, University of Maryland, College Park, USA for the support provided during the period of this work. He would also like to express his gratitude to the Norbert Wiener Center for Harmonic Analysis and Applications at the University of Maryland, College Park for its kind hospitality, and the Indian Institute of Technology Guwahati, India for its support. Further, the author thanks the anonymous referee for valuable suggestions which helped to improve the paper.

\bibliographystyle{plain}

\end{document}